\documentclass[12pt]{amsart}

\usepackage{amsmath,bm,amssymb,amsthm,textcomp}
\usepackage{enumitem}
\usepackage{mathtools}
\usepackage{hyperref}
\usepackage[numbers,sort&compress]{natbib}
\usepackage{tikz-cd}
\usetikzlibrary{cd}
\usepackage{comment}

\DeclareMathOperator{\hdim}{\dim_H}
\DeclareMathOperator{\mul}{mul}
\DeclareMathOperator{\cl}{cl}

\theoremstyle{plain}
\newtheorem{theorem}{Theorem}[section]
\newtheorem{lemma}[theorem]{Lemma}
\newtheorem{proposition}[theorem]{Proposition}
\newtheorem{corollary}[theorem]{Corollary}

\newtheorem*{claim}{Claim}

\theoremstyle{definition}
\newtheorem{definition}[theorem]{Definition}

\theoremstyle{remark}
\newtheorem{remark}[theorem]{Remark}
\newtheorem{remarks}[theorem]{Remarks}

\counterwithin{equation}{section}

\allowdisplaybreaks

\begin{document}

\title{Exceptional sets to Shallit's law of leap years in Pierce expansions}

\author{Min Woong Ahn}
\address{Department of Mathematics, SUNY at Buffalo, Buffalo, NY 14260-2900, USA}
\email{minwoong@buffalo.edu}
\curraddr{\sc Department of Mathematics Education, Silla University, 140, Baegyang-daero 700beon-gil, Sasang-gu, Busan, 46958, Republic of Korea}
\email{minwoong@silla.ac.kr}

\date{\today}


\begin{abstract}
In his 1994 work, Shallit introduced a rule for determining leap years that generalizes both the historically used Julian calendar and the contemporary Gregorian calendar. This rule depends on a so-called intercalation sequence. According to what we term Shallit's law of leap years, almost every point of the interval $[0,1]$ with respect to the Lebesgue measure has the same limsup and liminf, respectively, of a quotient defined in terms of the number of leap years determined by the rule using the Pierce expansion digit sequence as an intercalation sequence. In this paper, we show that the set of exceptions to this law is dense and has full Hausdorff dimension in $[0,1]$, and that the exceptional set intersected with any non-empty open subset of $[0,1]$ has full Hausdorff dimension in $[0,1]$. As a more general result, we establish that for certain subsets of $[0,1]$ concerning the limiting behavior of Pierce expansion digits, intersecting with a non-empty open subset of $[0,1]$ preserves the Hausdorff dimension.
\end{abstract}

\maketitle

\tableofcontents

\section{Introduction} \label{Introduction}

A calendar year consists of either $365$ or $366$ days, and a year with $366$ days is called a {\em leap year}. In practice, the more common occurrence is a non-leap year, easily observed by an ordinary person without a thorough understanding of calendrical systems. This discrepancy is attributed to the fact that the length of a {\em tropical year}, one of the various definitions of the Earth's orbital period around the Sun that calendars aim to match, averages approximately $365.242189$ days (see \cite[Chapter 14]{RD18}). It falls between precisely $365$ and $366$ days but is closer to $365$.

The calendar commonly used today is known as the Gregorian calendar. The rule for designating a leap year in the Gregorian calendar is essentially based on the inclusion-exclusion principle in basic set theory. To elaborate, for constructing the set of leap years, we initially include all years that are multiples of $4$, then exclude those that are multiples of $100$, and finally include those that are multiples of $400$. To put it formally, a year $N$ is a leap year in the Gregorian calendar if and only if
\[
\sum_{k=1}^3 (-1)^{k+1} \mul (N, \sigma_1 \dotsm \sigma_k) = 1,
\]
where $(\sigma_1, \sigma_2, \sigma_3) \coloneqq (4, 25, 4)$ is a finite sequence comprising $3$ terms, and $\mul \colon \mathbb{N}^2 \to \{ 0, 1 \}$ is a function defined as
\[
\mul (m,n) \coloneqq \begin{cases} 1, &\text{if $m$ is an integer multiple of $n$}; \\ 0, &\text{otherwise}, \end{cases}
\]
for each $(m,n) \in \mathbb{N}^2$. For instance, we have
\begin{align*}
\sum_{k=1}^3 (-1)^{k+1} \mul (2028, \sigma_1 \dotsm \sigma_k) &= 1 + 0 + 0 = 1, \\
\sum_{k=1}^3 (-1)^{k+1} \mul (2100, \sigma_1 \dotsm \sigma_k) &= 1 + (-1) + 0 = 0, \\
\sum_{k=1}^3 (-1)^{k+1} \mul (2400, \sigma_1 \dotsm \sigma_k) &= 1 + (-1) + 1 = 1,
\end{align*}
and so, if we adhere to the Gregorian calendar, the years $2028$ and $2400$ will be leap years, while the year $2100$ will not be.

A brief mention of the Julian calendar, upon which the reformation to the Gregorian calendar was built, is in order. The Julian calendar designates a year $N$ as a leap year if
\[
\sum_{k=1}^1 (-1)^{k+1} \mul (N, \tau_1 \dotsm \tau_k) = \mul (N, \tau_1) = 1,
\]
where $(\tau_1) \coloneqq (4)$ is a finite sequence consisting of $1$ term. In other words, in the Julian calendar, leap years are exactly the years that are multiples of $4$. It should be noted that both the Julian and Gregorian calendars approximate the tropical year, but do not precisely reflect it, as there are $365+(1/4)=365.25$ and $365+(1/4)-(1/100)+(1/400) = 365.2425$ days in a calendar year on average in each calendar, respectively. Consequently, scholars in related fields, such as mathematics and astronomy, have suggested various calendars for better approximations. For a closer look at mathematical aspects or historical expositions of calendrical systems and their reformation, we recommend the interested reader to refer to the comprehensive book \cite{RD18} and short articles \cite{Dut88, Eis12, Ric85, Sha94, Swe74, Swe86}.

A generalized rule for determining leap years, which extends both the Julian and Gregorian calendars, was proposed  by Shallit \cite{Sha94}. To begin, we fix an {\em intercalation sequence}, a finite or infinite sequence of positive integers $(\sigma_k)_{k \geq 1}$ satisfying $\sigma_k \geq 2$ for all $k \geq 2$. Then, the year $N$ is declared to be a leap year if
\[
\sum_{k \geq 1} (-1)^{k+1} \mul (N, \sigma_1 \dotsm \sigma_k) = 1,
\]
where the sum is taken over the domain on which the intercalation sequence is defined.

Suppose, for simplicity, that a tropical year has exactly 
\begin{align} \label{definition of eta}
365 + \eta 
\end{align}
days for some $\eta \in [0,1]$. (This assumption does not hold in reality because the length of a tropical year is not constant. On average, the value of $\eta$ is approximately $0.242189$, as mentioned in the first paragraph.) For each $N \in \mathbb{N}$, let $L(\sigma, N)$ denote the number of leap years from year $1$ through year $N$ determined by the generalized rule that uses $\sigma \coloneqq (\sigma_k)_{k \geq 1}$ as the intercalation sequence, that is, 
\[
L(\sigma, N) \coloneqq \# \left\{ m \in \{ 1, \dotsc, N \} : \sum_{k \geq 1} (-1)^{k+1} \mul (m, \sigma_1 \dotsm \sigma_k) = 1 \right\} ,
\]
where $\#$ denotes the cardinality of a set. Contrary to the impracticality of the definition, there is a practical formula (\cite[Theorem 1]{Sha94})
\begin{align} \label{L N formula}
L(\sigma, N) = \sum_{k \geq 1} (-1)^{k+1} \left\lfloor \frac{N}{\sigma_1 \dotsm \sigma_k} \right\rfloor,
\end{align}
which again reminds us of the inclusion-exclusion principle. Note that $N \eta$ is the theoretical number, not necessarily an integer, of leap years required up to the year $N$ to keep the calendar synchronized with the actual date based on the tropical year. Therefore, for the generalized determination rule to be acceptable in the long run, a crucial requirement is that the difference $|N\eta-L(\sigma, N)|$ should remain relatively small as $N$ grows. Moreover, unless we have the precise value of $\eta$ at hand, another reasonable demand might be that the existence of an intercalation sequence $\sigma = \sigma(x)$ such that the difference $|Nx-L(\sigma, N)|$ is small for most values of $x \in [0,1]$.

In this paper, we will consider the Pierce expansion digit sequence as a candidate for an intercalation sequence. As is well known, the {\em Pierce expansion} expresses $x \in [0,1]$ in the form
\begin{align} \label{Pierce expansion}
\begin{aligned}
x 
&= \langle d_1(x), d_2(x), d_3(x), \dotsc \rangle_P \\
&\coloneqq \sum_{k \in \mathbb{N}} \frac{(-1)^{k+1}}{d_1(x) \dotsm d_k(x)} 
= \frac{1}{d_1(x)} - \frac{1}{d_1(x) d_2(x)} + \frac{1}{d_1(x) d_2(x) d_3(x)} - \dotsb,
\end{aligned}
\end{align}
with the conventions $\infty \cdot \infty = \infty$ and $1/\infty = 0$, where the {\em digit} sequence $(d_k (x))_{k \in \mathbb{N}}$ is an $(\mathbb{N} \cup \{ \infty \})$-valued sequence defined as follows. Let $d_1 \colon [0,1] \to \mathbb{N} \cup \{ \infty \}$ and $T \colon [0,1] \to [0,1]$ be mappings given by
\[
d_1(x) \coloneqq \begin{cases} \lfloor 1/x \rfloor, &\text{if } x \neq 0; \\ \infty, &\text{if } x = 0, \end{cases}
\quad \text{and} \quad
T(x) \coloneqq \begin{cases} 1 - d_1(x) x, &\text{if } x \neq 0; \\ 0, &\text{if } x = 0, \end{cases}
\]
respectively. Now, define $d_k(x) \coloneqq d_1(T^{k-1}(x))$ for each $k \in \mathbb{N}$. One classical result in Pierce expansions states that the set of all strictly increasing positive integer-valued sequences and the set of irrationals in $[0,1]$ are in bijective correspondence via the Pierce expansion. In the case where $x \in [0,1]$ is rational, the digit sequence consists of finitely many, say $n \in \mathbb{N}$, strictly increasing $\mathbb{N}$-valued terms, followed by $\infty$'s, and satisfies $d_{n-1}(x) + 1 < d_n(x)$ whenever $n \geq 2$. We will present some basic facts related to Pierce expansions in detail in Section \ref{Preliminaries}. For further exploration of basic notions and deeper results in Pierce expansions, we refer the reader to \cite{Ahn23a, Ahn24, Ahn23c, Fan15, Pie29, Sch95, Sha86, Sha94}.

Notice that the formula \eqref{L N formula} and the series expansion \eqref{Pierce expansion} are of a similar form, particularly if we take the intercalation sequence as $(d_k(x))_{k \in \mathbb{N}}$ in \eqref{L N formula} and multiply \eqref{Pierce expansion} by $N$. Regarding the difference between these two particular quantities, what we refer to as {\em Shallit's law of leap years} (\cite[Theorem 3]{Sha94}) states that
\begin{align} \label{law of leap years}
\begin{aligned}
&\limsup_{N \to \infty} \frac{N x - L((d_k(x))_{k \in \mathbb{N}}, N)}{\sqrt{\log N}} = \frac{1}{\sqrt{2}}
\quad \text{and} \\
&\hspace{4cm} \liminf_{N \to \infty} \frac{N x - L((d_k(x))_{k \in \mathbb{N}}, N)}{\sqrt{\log N}} = -\frac{1}{\sqrt{2}}
\end{aligned}
\end{align}
for Lebesgue-almost every $x \in [0,1]$. Based on this law, he concluded the paper by pointing out that the calendar adopting the generalized rule, with its intercalation sequence given by the Pierce expansion digit sequence, aligns well with the actual dates over the long term. This is because, although $\sqrt{(\log N)/2}$ grows without bound as $N$ tends to infinity, its growth rate is relatively slow.

In this paper, our main focus is on the denseness and Hausdorff dimension of the set of exceptions to Shallit's law of leap years. For our investigation, we shall use the following first main result of this paper, which pertains to a set satisfying a certain property described in terms of Pierce expansions. We say that a subset $F$ of $[0,1]$ is {\em finite replacement-invariant} if the implication
\[
x \in F \implies \langle \tau_1, \dotsc, \tau_n, d_{n+1}(x), d_{n+2}(x), \dotsc \rangle_P \in F
\]
holds for any $n \in \mathbb{N}$ and finite sequence $(\tau_k)_{k=1}^n \in (\mathbb{N} \cup \{ \infty \})^n$ such that the expression
\[
\langle \tau_1, \dotsc, \tau_n, d_{n+1}(x), d_{n+2}(x), \dotsc \rangle_P
\]
is well-defined, i.e., there exists an $y \in [0,1]$ such that
\[
d_k (y) = 
\begin{cases}
\tau_k, &\text{if } k \in \{ 1, \dotsc, n \}; \\
d_k(x), &\text{otherwise}.
\end{cases}
\]

\begin{theorem} \label{intersection with open set theorem}
Let $F$ be a finite replacement-invariant subset of $[0,1]$. Then, the following hold.
\begin{enumerate}[label=\upshape(\roman*), ref=\roman*, leftmargin=*, widest=ii]
\item \label{intersection with open set theorem 1}
If $F$ is uncountable, then $F$ is dense in $[0,1]$. In particular, if $F$ has non-zero Hausdorff dimension, then $F$ is dense in $[0,1]$.
\item \label{intersection with open set theorem 2}
For any non-empty open subset $U$ of $[0,1]$, the intersection of $U$ with $F$ has the same Hausdorff dimension as $F$.
\end{enumerate}
\end{theorem}

\begin{remarks}
\begin{enumerate}[label=\upshape(\arabic*), ref=\arabic*, leftmargin=*, widest=2]
\item
Not all dense subsets of $[0,1]$ are finite replacement-invariant. For instance, the set $F \coloneqq [0,1] \setminus \{ \langle 2, 3, 4, \dotsc \rangle_P \}$ is dense in $[0,1]$ but not finite replacement-invariant, as $\langle 1, 3, 4, \dotsc \rangle_P \in F$ but $\langle 2, 3, 4, \dotsc, \rangle_P \not \in F$.
\item
In general, the fact that a subset $F$ of $[0,1]$ has non-zero Hausdorff dimension does not imply that $F$ is dense in $[0,1]$. For example, the interval $[0,1/2]$ has Hausdorff dimension $1$, but it is not dense in $[0,1]$.
\end{enumerate}
\end{remarks}

As an application of Theorem \ref{intersection with open set theorem}, we examine the set, the Hausdorff dimension of which was determined in \cite{Ahn24}, defined as follows. For each $\alpha \in [0, \infty]$, let
\begin{align} \label{definition of A alpha}
A (\alpha) \coloneqq \left\{ x \in [0,1] : \lim_{n \to \infty} \frac{\log d_n(x)}{n} = \alpha \right\}.
\end{align}
The second part of the following corollary generalizes \cite[Corollary 1.3]{Ahn24} (see Proposition \ref{hdim A alpha} below), and the corollary will play one of the two main roles in proving Theorem \ref{leap year theorem}.

\begin{corollary} \label{LLN corollary}
For the sets $A(\alpha)$, the following hold for each $\alpha \in [0, \infty]$.
\begin{enumerate}[label=\upshape(\roman*), ref=\roman*, leftmargin=*, widest=ii]
\item \label{LLN corollary 1}
$A(\alpha)$ is dense in $[0,1]$.
\item \label{LLN corollary 2}
For any non-empty open subset $U$ of $[0,1]$, the set $U \cap A(\alpha)$ has Hausdorff dimension $1$.
\end{enumerate}
\end{corollary}

Now, we are in a position to state the second main result, which is concerned with certain exceptional sets to Shallit's law of leap years. Specifically, we consider the set
\begin{align} \label{definition of S beta}
\begin{aligned}
&S (\alpha) \coloneqq \left\{ x \in [0,1] : \limsup_{N \to \infty} \frac{Nx - L((d_k(x))_{k \in \mathbb{N}}, N)}{\sqrt{\log N}} = \frac{1}{\sqrt{2\alpha}} \right. \\
&\hspace{5cm} \left. \text{ and } \liminf_{N \to \infty} \frac{Nx - L((d_k(x))_{k \in \mathbb{N}}, N)}{\sqrt{\log N}} = - \frac{1}{\sqrt{2\alpha}} \right\}
\end{aligned}
\end{align}
for each $\alpha \in [0, \infty]$, where we adopt the conventions $1/0 = \infty$ and $1/\infty = 0$.

\begin{theorem} \label{leap year theorem}
For the sets $S(\alpha)$, the following hold for each $\alpha \in [0, \infty]$.
\begin{enumerate}[label=\upshape(\roman*), ref=\roman*, leftmargin=*, widest=ii]
\item \label{leap year theorem 1}
$S(\alpha)$ is dense in $[0,1]$.
\item \label{leap year theorem 2}
For any non-empty open subset $U$ of $[0,1]$, the set $U \cap S(\alpha)$ has Hausdorff dimension $1$.
\end{enumerate}
\end{theorem}

\begin{corollary} \label{leap year corollary}
The set of exceptions to \eqref{law of leap years} is dense in $[0,1]$ and has Hausdorff dimension $1$.
\end{corollary}

Our results suggest that, for any given non-empty open subset of $[0,1]$, while the generalized rule, using the Pierce expansion digit sequence as the intercalation sequence, works well for most points in the Lebesgue measure sense (by the law of leap years), considering the Hausdorff dimension reveals that there are still plenty of points where the difference $|Nx-L((d_k(x))_{k \in \mathbb{N}}, N)|$ grows quite rapidly as $N$ increases. In particular, no matter how short the interval approximating the value of $\eta$ in \eqref{definition of eta}---such as $(0.2,0.3)$ or $(0.24, 0.25)$, both containing $0.242189$---there are substantial possibilities, in the Hausdorff dimension sense, for the calendar to deviate significantly from the actual date.

This paper is organized as follows. In Section \ref{Preliminaries}, we present some facts and results on the Hausdorff dimension and Pierce expansions. Section \ref{Auxiliary results} contains useful lemmas that will be employed in establishing the main results. Finally, we prove the main results in Section \ref{Proofs of main results}.

Throughout the paper, the closed unit interval $[0,1]$ will be endowed with the usual subspace topology inherited from $\mathbb{R}$. For a subset $F$ of $[0,1]$, we denote by $\hdim F$ its Hausdorff dimension. The set of positive integers will be denoted by $\mathbb{N}$, the set of extended positive integers by $\mathbb{N}_\infty \coloneqq \mathbb{N} \cup \{ \infty \}$, and the set of irrational numbers in $[0,1]$ by $\mathbb{I} \coloneqq [0,1] \setminus \mathbb{Q}$. Following the convention, we define $c \cdot \infty \coloneqq \infty$, $\infty^c \coloneqq \infty$, and $c / 0 \coloneqq \infty$ for any $c \in (0, \infty)$, and $c/\infty \coloneqq 0$ and $\infty \pm c \coloneqq \infty$ for any $c \in \mathbb{R}$.

\section{Preliminaries} \label{Preliminaries}

This section is dedicated to presenting some facts and results on the Hausdorff dimension and Pierce expansions.

We first provide the definition of the Hausdorff dimension.

\begin{definition} [See {\cite[Chapter 3]{Fal14}}] \label{hdim definition}
For a subset $F$ of $[0,1]$, the {\em Hausdorff dimension} of $F$ is defined by
\[
\hdim F \coloneqq \inf \{ s \geq 0 : \mathcal{H}^s (F) = 0 \} = \sup \{ s : \mathcal{H}^s (F) = \infty \},
\]
considering the supremum of the empty set to be $0$. Here, $\mathcal{H}^s(F)$, for $s \in [0, \infty)$, is the {\em $s$-dimensional Hausdorff measure} given by
\[
\mathcal{H}^s (F) \coloneqq \lim_{\delta \to 0} \left( \inf \left\{ \sum_{k \in \mathbb{N}} |U_k|^s : F \subseteq \bigcup_{k \in \mathbb{N}} U_k \text{ and } |U_k| \in (0, \delta] \text{ for each } k \in \mathbb{N} \right\} \right),
\]
where $| \cdot |$ denotes the diameter of a set.
\end{definition}

{\em Monotonicity} and {\em countable stability} are two important properties of the Hausdorff dimension that we will frequently utilize.

\begin{proposition} [See {\cite[pp.~48--49]{Fal14}}] \label{monotonicity and countable stability}
For any subsets $E$ and $F$ of $[0,1]$, the following hold.
\begin{enumerate}[label=\upshape(\roman*), ref=\roman*, leftmargin=*, widest=ii]
\item \label{monotonicity and countable stability 1}
If $E \subseteq F$, then $\hdim E \leq \hdim F$.
\item \label{monotonicity and countable stability 2}
If $F = \bigcup_{k \in \mathbb{N}} F_k$, then $\hdim F = \sup_{k \in \mathbb{N}} \{ \hdim F_k \}$.
\end{enumerate}
\end{proposition}

Another useful property of the Hausdorff dimension is its {\em bi-Lipschitz invariance}.

\begin{proposition} [See {\cite[Proposition 3.3]{Fal14}}] \label{bi-Lipschitz invariance}
Let $F$ be a subset of $[0,1]$. If $g \colon F \to \mathbb{R}$ is a bi-Lipschitz mapping, then $\hdim F = \hdim g(F)$.
\end{proposition}

Now, we introduce a sequence set consisting of certain sequences in $\mathbb{N}_\infty$. This sequence set was discussed in detail in \cite{Ahn23a}, where we studied the error-sum function of Pierce expansions. Let
\[
\Sigma_0 \coloneqq \{ (\sigma_k)_{k \in \mathbb{N}} \in \{ \infty \}^{\mathbb{N}} \} = \{ (\infty, \infty, \dotsc) \}.
\]
For each $n \in \mathbb{N}$, define
\[
\Sigma_n \coloneqq \{ (\sigma_k)_{k \in \mathbb{N}} \in \mathbb{N}^{\{ 1, \dotsc, n \}} \times \{ \infty \}^{\mathbb{N} \setminus \{ 1, \dotsc, n \}} : \sigma_1 < \dotsb < \sigma_n \}.
\]
For brevity, we write $(\sigma_1, \dotsc, \sigma_n)$ for $(\sigma_1, \dotsc, \sigma_n, \infty, \infty, \dotsc) \in \Sigma_n$, provided that the context is clear. For the set of all strictly increasing infinite sequences in $\mathbb{N}$, put
\[
\Sigma_\infty \coloneqq \{ (\sigma_k)_{k \in \mathbb{N}} \in \mathbb{N}^{\mathbb{N}} : \sigma_k < \sigma_{k+1} \text{ for all } k \in \mathbb{N} \}.
\]
Finally, let
\begin{align*}
\Sigma &\coloneqq \Sigma_0 \cup \bigcup_{n \in \mathbb{N}} \Sigma_n \cup \Sigma_\infty, \\
\Sigma' &\coloneqq \Sigma \setminus \{ (\sigma_k)_{k=1}^n \in \Sigma : n \geq 2 \text{ and } \sigma_n = \sigma_{n-1} + 1 \}.
\end{align*}

For each $n \in \mathbb{N}$ and $\sigma \coloneqq (\sigma_k)_{k=1}^n \in \Sigma_n$, we define the {\em cylinder set} associated with $\sigma$ by
\[
\Upsilon_{\sigma} \coloneqq \{ (\tau_k)_{k \in \mathbb{N}} \in \Sigma : \tau_k = \sigma_k \text{ for all } k \in \{ 1, \dotsc, n \} \}.
\]
Similarly, the {\em fundamental interval} associated with $\sigma$ is defined by
\[
I_{\sigma} \coloneqq f^{-1} (\Upsilon_\sigma) = \{ x \in [0,1] : d_k(x) = \sigma_k \text{ for all } k \in \{ 1, \dotsc, n \} \},
\]
i.e., $I_{\sigma}$ consists of numbers in $[0,1]$ whose digit sequence begins with $\sigma_1, \dotsc, \sigma_n$.

For each $x \in [0,1]$, we denote its Pierce expansion digit sequence $(d_k(x))_{k \in \mathbb{N}}$ by $f(x)$, that is, $f \colon [0,1] \to \Sigma$ is a map defined by $x \mapsto (d_k(x))_{k \in \mathbb{N}}$ for each $x \in [0,1]$. Conversely, define a map $\varphi \colon \Sigma \to [0,1]$ by
\[
\varphi (\sigma) \coloneqq \sum_{k \in \mathbb{N}} \frac{(-1)^{k+1}}{\sigma_1 \dotsm \sigma_k}
\]
for each $\sigma \coloneqq (\sigma_k)_{k \in \mathbb{N}} \in \Sigma$. Note that, for any $x \in [0,1]$, we have
\[
(\varphi \circ f)(x) = \sum_{k \in \mathbb{N}} \frac{(-1)^{k+1}}{d_1(x) \dotsm d_k(x)} = \langle d_1(x), d_2(x), \dotsc \rangle_P = x
\]
by definition.

Let $\mathbb{N}_\infty$ denote the one-point compactification of the discrete space $\mathbb{N}$. The product space $\mathbb{N}_\infty^{\mathbb{N}}$ is compact by Tychonoff's theorem. Equip $\Sigma \subseteq \mathbb{N}_\infty^{\mathbb{N}}$ with the subspace topology. With the notations introduced so far, a classical result in Pierce expansions mentioned in Section \ref{Introduction}---the one-to-one correspondence between $\mathbb{I}$ and $\Sigma_\infty$---can be restated and extended as follows.

\begin{proposition} [See {\cite[Section 3.2]{Ahn23a}}] \label{f is homeo}
For the mappings $f \colon [0,1] \to \Sigma$ and $\varphi \colon \Sigma \to [0,1]$, the following hold.
\begin{enumerate}[label=\upshape(\roman*), ref=\roman*, leftmargin=*, widest=ii]
\item \label{f is homeo 1}
$f|_{\mathbb{I}} \colon \mathbb{I} \to \Sigma_\infty$, the restriction of $f$ to $\mathbb{I}$, is a homeomorphism with the continuous inverse $\varphi|_{\Sigma_\infty} : \Sigma_\infty \to \mathbb{I}$, the restriction of $\varphi$ to $\Sigma_\infty$.
\item \label{f is homeo 2}
$f \colon [0,1] \to \Sigma'$ is a bijection with the inverse $\varphi|_{\Sigma'} \colon \Sigma' \to [0,1]$, the restriction of $\varphi$ to $\Sigma'$.
\end{enumerate}
\end{proposition}

The following proposition precisely describes the fundamental intervals.

\begin{proposition}[See {\cite[Theorem 1]{Sha86}} and {\cite[Proposition 3.5]{Ahn23a}}] \label{I sigma}
Let $n \in \mathbb{N}$ and $\sigma \coloneqq (\sigma_k)_{k=1}^n \in \Sigma_n$. Define $\sigma' \in \Sigma_n$ by
\[
\sigma' \coloneqq (\underbrace{\sigma_1, \dotsc, \sigma_{n-1}}_{\text{$n-1$ terms}}, \sigma_n+1).
\]
Then, $I_{\sigma}$ is a subinterval of $[0,1]$ with endpoints $\varphi (\sigma)$ and $\varphi (\sigma')$; more precisely,
\[
I_{\sigma}  = \begin{cases}
(\varphi (\sigma'), \varphi (\sigma)], &\text{if $n$ is odd}; \\
[\varphi (\sigma), \varphi (\sigma')), &\text{if $n$ is even},
\end{cases}
\quad \text{or} \quad
I_{\sigma}  = \begin{cases}
(\varphi (\sigma'), \varphi (\sigma)), &\text{if $n$ is odd}; \\
(\varphi (\sigma), \varphi (\sigma')), &\text{if $n$ is even},
\end{cases}
\]
according as $\sigma \in \Sigma'$ or $\sigma \not \in \Sigma'$.
\end{proposition}

\begin{proposition} [See {\cite[Lemma 3.4 and Remark 3.5]{Ahn24}}] \label{shift of digits proposition}
Let $n \in \mathbb{N}$ and $\sigma \coloneqq (\sigma_k)_{k=1}^n \in \Sigma_n$. Define a linear map $g_\sigma \colon [0,1] \to g_\sigma ([0,1])$ by
\begin{align*}
g_\sigma (x) \coloneqq \varphi (\sigma) +  \frac{(-1)^n}{\sigma_1 \dotsm \sigma_n} x = \frac{1}{\sigma_1} - \frac{1}{\sigma_1 \sigma_2} + \dotsb + \frac{(-1)^{n+1}}{\sigma_1 \dotsm \sigma_n} + \frac{(-1)^n}{\sigma_1 \dotsm \sigma_n} x
\end{align*}
for each $x \in [0,1]$, and denote by $g_\sigma^{-1}$ its inverse map, which is also linear. Then, the following hold.
\begin{enumerate}[label=\upshape(\roman*), ref=\roman*, leftmargin=*, widest=ii]
\item \label{shift of digits proposition 1}
For each $x \in \mathbb{I}$, if $\sigma_n < d_1(x)$, then
\[
g_\sigma(x) = \langle \sigma_1, \dotsc, \sigma_n, d_1(x), d_2(x), \dotsc \rangle_P \in I_\sigma \cap \mathbb{I}.
\]
\item \label{shift of digits proposition 2}
For each $y \in I_\sigma \cap \mathbb{I}$, we have $y \in g_\sigma ([0,1])$ and
\[
g_\sigma^{-1}(y) = \langle d_{n+1}(y), d_{n+2}(y), \dotsc \rangle_P \in \mathbb{I}.
\]
\end{enumerate}
\end{proposition}

Recall from \eqref{definition of A alpha} that the set $A(\alpha)$, defined for each $\alpha \in [0, \infty]$, is concerned with the growth rate of Pierce expansion digits. On this matter, one of the earliest historically significant results is the following, known as the {\em law of large numbers} in Pierce expansions.

\begin{proposition} [{\cite[Theorem 16]{Sha86}}] \label{law of large numbers}
The set $A(1)$ has full Lebesgue measure on $[0,1]$, i.e., for Lebesgue-almost every $x \in [0,1]$, we have
\[
\lim_{n \to \infty} \frac{\log d_n(x)}{n} = 1.
\]
\end{proposition}

Regarding some exceptional sets emerging from the law of large numbers in Pierce expansions, we have recently established the following result.

\begin{proposition} [{\cite[Corollary 1.3]{Ahn24}}] \label{hdim A alpha}
For each $\alpha \in [0, \infty]$, the set $A(\alpha)$ has full Hausdorff dimension in $[0,1]$, i.e.,
\[
\hdim \left\{ x \in [0,1] : \lim_{n \to \infty} \frac{\log d_n(x)}{n} = \alpha \right\} = 1.
\]
In other words, Corollary \ref{LLN corollary}(\ref{LLN corollary 2}) holds when $U = [0,1]$.
\end{proposition}

\section{Auxiliary results} \label{Auxiliary results}

In this section, we present auxiliary results that will be utilized in proving the main results.

We first present crucial properties of the fundamental intervals in the next two lemmas, both of which are consequences of Propositions \ref{f is homeo} and \ref{I sigma}.

\begin{lemma} \label{union of fundamental intervals}
We have
\begin{align} \label{partition of unit interval}
(0,1] = \bigcup_{j \in \mathbb{N}} I_{(j)}.
\end{align}
Furthermore, for any $n \in \mathbb{N}$ and $\sigma \coloneqq (\sigma_k)_{k=1}^n \in \Sigma_n$, we have
\[
I_\sigma =
\begin{cases}
\{ \varphi (\sigma) \} \cup \bigcup_{j \geq \sigma_n+1} I_{(\sigma_1, \dotsc, \sigma_n, j)}, &\text{if } \sigma \in \Sigma'; \\
\bigcup_{j \geq \sigma_n+1} I_{(\sigma_1, \dotsc, \sigma_n, j)}, &\text{if } \sigma \not \in \Sigma',
\end{cases}
\]
and, consequently,
\begin{align} \label{partition of I sigma}
I_\sigma \cap \mathbb{I} = \bigcup_{j \geq \sigma_n+1} [ I_{(\sigma_1, \dotsc, \sigma_n, j)} \cap \mathbb{I} ].
\end{align}
\end{lemma}

\begin{proof}
Since $\Sigma_0$ is a singleton containing the unique element $(\infty, \infty, \dotsc) = f(0) \in f([0,1])$, Proposition \ref{f is homeo}(\ref{f is homeo 2}) tells us that $f^{-1}(\Sigma_0) = \{ 0 \}$. Then, by definitions, we have
\[
\bigcup_{j \in \mathbb{N}} I_{(j)} = f^{-1} \left( \bigcup_{j \in \mathbb{N}} \Upsilon_{(j)} \right) = f^{-1}(\Sigma \setminus \Sigma_0) = [0,1] \setminus \{ 0 \} = (0,1].
\]
For the second assertion, fix $n \in \mathbb{N}$ and $\sigma \coloneqq (\sigma_k)_{k=1}^n \in \Sigma_n$. Then, we find, by definitions, that
\[
\bigcup_{j \geq \sigma_n+1} I_{(\sigma_1, \dotsc, \sigma_n, j)}
= f^{-1} \left( \bigcup_{j \geq \sigma_n+1} \Upsilon_{(\sigma_1, \dotsc, \sigma_n, j)} \right)
= f^{-1} (\Upsilon_\sigma \setminus \{ \sigma \}) = I_\sigma \setminus f^{-1}(\{ \sigma \}).
\]
Assume $\sigma \in \Sigma'$. Then, by Proposition \ref{f is homeo}(\ref{f is homeo 2}), we have $\sigma \in f([0,1])$ and $f^{-1}(\{ \sigma \}) = \{ \varphi (\sigma) \}$. But $\varphi (\sigma) \in I_\sigma$ by Proposition \ref{I sigma}, so that $[I_\sigma \setminus f^{-1}( \{ \sigma \})] \cup \{ \varphi (\sigma ) \} = I_\sigma$. On the other hand, if $\sigma \not \in \Sigma'$, then, in view of Proposition \ref{f is homeo}(\ref{f is homeo 2}), we have $\sigma \not \in f([0,1])$, i.e., $f^{-1} ( \{ \sigma \}) = \varnothing$. Hence, we have the expression for $I_\sigma$ as a countable union in the statement of the lemma. Lastly, note that, in any case, $f^{-1}(\{ \sigma \}) \subseteq \mathbb{Q}$. Therefore, \eqref{partition of I sigma} follows.
\end{proof}

\begin{lemma} \label{existence of I sigma lemma}
For any non-empty open subset $U$ of $[0,1]$, there exists a $\sigma \in \bigcup_{n \in \mathbb{N}} \Sigma_n$ such that $I_{\sigma} \subseteq U$.
\end{lemma}

\begin{proof}
Let $U$ be a non-empty open subset of $[0,1]$. Take a non-degenerate open interval $(a,b) \subseteq [0,1]$ contained in $U$. Choose another non-degenerate open interval $J \coloneqq (a',b')$ contained in $(a,b)$ such that $a<a'$ and $b'<b$. Then, $J \cap \mathbb{I}$ is open in $\mathbb{I}$, and hence, by Proposition \ref{f is homeo}(\ref{f is homeo 1}), $f (J \cap \mathbb{I})$ is open in $\Sigma_\infty$. By the subspace topology on $\Sigma_\infty$, we can find a $\sigma \in \bigcup_{n \in \mathbb{N}} \Sigma_n$ such that $\Upsilon_\sigma \cap \Sigma_\infty \subseteq f (J \cap \mathbb{I})$. Hence, 
\[
I_\sigma \cap \mathbb{I} = f^{-1}(\Upsilon_\sigma) \cap f^{-1}(\Sigma_\infty) = f^{-1} (\Upsilon_\sigma \cap \Sigma_\infty) \subseteq J \cap \mathbb{I} \subseteq J.
\]
Recall from Proposition \ref{I sigma} that $I_\sigma$ is a subinterval of $[0,1]$. It then follows that 
\[
I_\sigma \subseteq \cl (I_\sigma) = \cl( I_\sigma \cap \mathbb{I}) \subseteq \cl(J) = [a',b'] \subseteq (a,b) \subseteq U,
\]
where $\cl( \cdot)$ denotes the closure in $[0,1]$. This completes the proof of the lemma.
\end{proof}

The following lemma characterizes a dense subset of $[0,1]$ in terms of Pierce expansions.

\begin{lemma} \label{prevalence lemma}
Let $E$ be a subset of $[0,1]$. Then, $E$ is dense in $[0,1]$ if and only if $I_\sigma \cap E \neq \varnothing$ for any $\sigma \in \bigcup_{n \in \mathbb{N}} \Sigma_n$.
\end{lemma}

\begin{proof}
Suppose first that $E$ is dense in $[0,1]$. Let $\sigma \in \bigcup_{n \in \mathbb{N}} \Sigma_n$ be arbitrary. Since $I_\sigma$ is a non-degenerate subinterval of $[0,1]$ by Proposition \ref{I sigma}, we can find a non-degenerate open interval $J$ contained in $I_\sigma$. Clearly, $J = J \cap [0,1]$ is open in $[0,1]$. Then, $J \cap E \neq \varnothing$ by the denseness hypothesis, and thus, $I_\sigma \cap E \neq \varnothing$.

Conversely, suppose that $I_\sigma \cap E \neq \varnothing$ for any $\sigma \in \bigcup_{n \in \mathbb{N}} \Sigma_n$. Let $U$ be a non-empty open subset of $[0,1]$. By Lemma \ref{existence of I sigma lemma}, we can find a $\tau \in \bigcup_{n \in \mathbb{N}} \Sigma_n$ satisfying $I_\tau \subseteq U$. Since $I_\tau \cap E \neq \varnothing$ by the hypothesis, we infer that $U \cap E \neq \varnothing$. This proves that $E$ is dense in $[0,1]$.
\end{proof}

\begin{lemma} \label{I sigma and open set equivalence lemma}
Let $E$ be a subset of $[0,1]$. The following statements are equivalent.
\begin{enumerate}[label=\upshape(\roman*), ref=\roman*, leftmargin=*, widest=2]
\item \label{I sigma and open set equivalence lemma 1}
$\hdim (I_\sigma \cap E) = \hdim E$ for any $\sigma \in \bigcup_{n \in \mathbb{N}} \Sigma_n$.
\item \label{I sigma and open set equivalence lemma 2}
$\hdim (U \cap E) = \hdim E$ for any non-empty open subset $U$ of $[0,1]$.
\end{enumerate}
\end{lemma}

\begin{proof}
The proof is similar to that of the preceding lemma, so we omit the details. However, it additionally requires the use of the monotonicity property of the Hausdorff dimension (Proposition \ref{monotonicity and countable stability}(\ref{monotonicity and countable stability 1})).
\end{proof}

\begin{lemma} \label{Z c is countable lemma}
Let $c \in [0, \infty)$. For each $M \in \mathbb{N}$, let
\begin{align} \label{definition of Z c}
Z_c^{(M)} \coloneqq \{ x \in \mathbb{I} : d_n(x) \leq n+c \text{ for all } n \geq M \}.
\end{align}
Then, for each $M \in \mathbb{N}$, $Z_c^{(M)}$ is equal to $Z_c^{(1)}$. Further, $Z_c^{(M)}$ is countable, hence is of Hausdorff dimension $0$.
\end{lemma}

\begin{proof}
Let $M \in \mathbb{N} \setminus \{ 1 \}$. By definition, we clearly have $Z_c^{(M)} \supseteq Z_c^{(1)}$. To prove the reverse inclusion, suppose $x \in Z_c^{(M)}$. Then, $(d_n(x))_{n \in \mathbb{N}} \in \Sigma_\infty$ by Proposition \ref{f is homeo}(\ref{f is homeo 1}) so that $d_n (x) < d_{n+1} (x) < \infty$ for all $n \in \mathbb{N}$. Hence, we find that
\begin{align*}
d_1(x) 
\leq d_2(x)-1 \leq \dotsb \leq d_{M-1}(x)-(M-2) &\leq d_M(x)-(M-1) \\
&\leq M+c-(M-1) = 1+c,
\end{align*}
where the inequality in the second line holds true by the hypothesis $x \in Z_c^{(M)}$. From this, we see that $d_n(x) \leq n+c$ for all $n \in \{ 1, \dotsc, M-1 \}$, and thus for all $n \in \mathbb{N}$. Therefore, $x \in Z_c^{(1)}$, and this establishes the equality $Z_c^{(M)} = Z_c^{(1)}$.

Now, by the preceding paragraph, to finish the proof of the lemma, we only need to show that $Z_c^{(1)}$ is countable. It will then imply that $\hdim Z_c^{(1)}=0$, since any countable set has Hausdorff dimension $0$ (see \cite[p.~49]{Fal14}). By Proposition \ref{f is homeo}(\ref{f is homeo 1}), it is then sufficient to show that the set
\[
f(Z_c^{(1)}) = \{ (\sigma_n)_{n \in \mathbb{N}} \in \Sigma_\infty : \sigma_n \leq n+c \text{ for all } n \in \mathbb{N} \}
\]
is countable. If $c \in [0,1)$, then $f(Z_c^{(1)})$ is a singleton, since any $(\sigma_n)_{n \in \mathbb{N}} \in f(Z_c^{(1)})$ satisfies $n \leq \sigma_n \leq n+c < n+1$, so that $\sigma_n = n$, for all $n \in \mathbb{N}$. Assume $c \in [1, \infty)$. Put $\theta (\sigma, n) \coloneqq \sigma_n - n$ for each $\sigma \coloneqq (\sigma_n)_{n \in \mathbb{N}} \in f (Z_c^{(1)})$ and $n \in \mathbb{N}$. Note that for fixed $\sigma \in f (Z_c^{(1)})$, the map $\theta (\sigma, n)$ is non-decreasing in $n$. Indeed, given $\sigma \coloneqq (\sigma_n)_{n \in \mathbb{N}} \in f(Z_c^{(1)})$, we have
\[
\theta(\sigma,n+1) - \theta(\sigma,n) = \sigma_{n+1} - \sigma_n - 1 \geq 1 - 1 = 0
\]
for each $n \in \mathbb{N}$. Hence, for each $\sigma \coloneqq (\sigma_n)_{n \in \mathbb{N}} \in f(Z_c^{(1)})$, we can find an ordered $\lfloor c \rfloor$-tuple
\[
(n_1, n_2, \dotsc, n_{\lfloor c \rfloor}) \in \{ (n_k)_{k=1}^{\lfloor c \rfloor} \in \mathbb{N}^{\lfloor c \rfloor} : n_1 \leq n_2 \leq \dotsb \leq n_{\lfloor c \rfloor} \}
\]
such that
\[
\theta (\sigma, n) = \sigma_n - n =
\begin{cases}
0, &\text{if } n_0 \coloneqq 1 \leq n < n_1; \\
1, &\text{if } n_1 \leq n < n_2; \\
\vdots \\
\lfloor c \rfloor-1, &\text{if } n_{\lfloor c \rfloor-1} \leq n < n_{\lfloor c \rfloor}; \\
\lfloor c \rfloor, &\text{if } n_{\lfloor c \rfloor} \leq n.
\end{cases}
\]
Observe that the choice of $(n_1, n_2, \dotsc, n_{\lfloor c \rfloor})$ is unique. To see this, suppose that we have two distinct choices $(n_1, n_2, \dotsc, n_{\lfloor c \rfloor})$ and $(n_1', n_2', \dotsc, n_{\lfloor c \rfloor}')$ for some $\sigma \in f(Z_c^{(1)})$. Put $k \coloneqq \min \{ 1 \leq j \leq \lfloor c \rfloor : n_j \neq n_j' \}$. Without loss of generality, we may assume $n_k < n_k'$. Then, $n_{k-1}' = n_{k-1} \leq n_k \leq n_k'-1 < n_k'$. Since $\theta (\sigma, n)$ is non-decreasing in $n$ for fixed $\sigma$, we deduce that
\[
k \leq \theta (\sigma, n_k) \leq \theta (\sigma, n_k'-1) = k-1,
\]
a contradiction. Hence, the mapping $\sigma \mapsto (n_1, n_2, \dotsc, n_{\lfloor c \rfloor})$ is well-defined on $f(Z_c^{(1)})$. Furthermore, it is immediate from the definition that this mapping is an injection from $f(Z_c^{(1)})$ into $\mathbb{N}^{\lfloor c \rfloor}$. But $\mathbb{N}^{\lfloor c \rfloor}$ is countable, and thus, we conclude that $f(Z_c^{(1)})$ is countable. Hence the lemma.
\end{proof}

Recall the definition of the set $A(\alpha)$, $\alpha \in [0, \infty]$, described in \eqref{definition of A alpha}. The following lemma generalizes \cite[Lemma 4]{Sha94}, where the case for $\alpha = 1$ was established.

\begin{lemma} \label{growth rate of log product lemma}
Let $x \in [0,1]$. If $x \in A(\alpha)$ for some $\alpha \in [0, \infty]$, then
\[
\lim_{n \to \infty} \frac{\log (d_1 (x) \dotsm d_n (x))}{n^2/2} = \alpha.
\]
\end{lemma}

\begin{proof}
Suppose first that $x \in A(\alpha)$ for some $\alpha \in [0, \infty)$. Let $\varepsilon \in (0,\infty)$ be arbitrary. Since $(\log d_n(x))/n \to \alpha$ as $n \to \infty$ by assumption, there exists a $K \in \mathbb{N}$ such that 
$(\alpha - \varepsilon) n < \log d_n(x) < (\alpha + \varepsilon) n$ for all $n > K$. Then, for any $n > K$, we have
\[
(\alpha-\varepsilon) \frac{(n-K)(n+K+1)}{2} < \log (d_{K+1}(x) \dotsm d_n(x)) < (\alpha+\varepsilon) \frac{(n-K)(n+K+1)}{2}.
\]
It follows that
\[
\alpha- \varepsilon \leq \liminf_{n \to \infty} \frac{\log (d_1(x) \dotsm d_n(x))}{n^2/2} \leq \limsup_{n \to \infty} \frac{\log (d_1(x) \dotsm d_n(x))}{n^2/2} \leq \alpha + \varepsilon.
\]
On letting $\varepsilon \to 0^+$, the desired limit follows.

The proof for the case where $x \in A(\infty)$ is similar; hence, we omit the details.
\end{proof}

It is known that the series $\sum_{n \in \mathbb{N}} (1/ d_n(x))$ converges Lebesgue-almost everywhere on $[0,1]$ (see {\cite[Theorem 12]{Sha86}} and {\cite[Corollary 1.17]{Ahn23c}}). This is an immediate consequence of the law of large numbers (Proposition \ref{law of large numbers}). That is to say, if $x$ is an element of a full-Lebesgue measure set $A(1)$, then $\sum_{n \in \mathbb{N}} (1/d_n(x))$ is convergent. We shall need the following generalization of this fact.

\begin{lemma} \label{sum of reciprocal lemma}
Let $x \in [0,1]$. If $x \in A(\alpha)$ for some $\alpha \in (0, \infty]$, then the series $\sum_{n \in \mathbb{N}} (1/d_n(x) )$ is convergent.
\end{lemma}

\begin{proof}
Suppose that $x \in A(\alpha)$ for some $\alpha \in (0, \infty]$. Since $(\log d_n(x))/n \to \alpha$, or, equivalently, $(d_n(x))^{1/n} \to e^\alpha \in (1, \infty]$, as $n \to \infty$, we see that $(1/d_n(x))^{1/n} \to 1/e^\alpha \in [0,1)$ as $n \to \infty$. Thus, the root test tells us that the series $\sum_{n \in \mathbb{N}} (1/d_n(x) )$ converges.
\end{proof}

\section{Proofs of main results} \label{Proofs of main results}

In this section, we prove the two main theorems and their corresponding corollaries mentioned in Section \ref{Introduction}.

\subsection{Proofs of Theorem \ref{intersection with open set theorem} and Corollary \ref{LLN corollary}}

\begin{proof} [Proof of Theorem \ref{intersection with open set theorem}]
Put $E \coloneqq F \cap \mathbb{I}$, which is clearly finite replacement-invariant. For part (\ref{intersection with open set theorem 1}), if we show that $E$ is dense in $[0,1]$, then it will imply that $F$ is also dense in $[0,1]$. For part (\ref{intersection with open set theorem 2}), we note that $\hdim (X \cap E) = \hdim (X \cap F)$ for any subset $X \subseteq [0,1]$. This equality follows directly from the simple fact that
\[
X \cap F = (X \cap F \cap \mathbb{Q}) \cup (X \cap F \cap \mathbb{I}) = (X \cap F \cap \mathbb{Q}) \cup (X \cap E),
\]
which, combined with countable stability (Proposition \ref{monotonicity and countable stability}(\ref{monotonicity and countable stability 2})), implies that
\begin{align*}
\hdim (X \cap F) 
&= \max \{ \hdim (X \cap F \cap \mathbb{Q}), \hdim (X \cap E) \} \\
&= \max \{ 0, \hdim (X \cap E) \} = \hdim (X \cap E),
\end{align*}
since $X \cap F \cap \mathbb{Q}$ is a countable set. Therefore, for a non-empty open subset $U$ of $[0,1]$, if the equality $\hdim (U \cap E) = \hdim E$ holds true, then we may conclude that $\hdim (U \cap F) = \hdim F$.

(\ref{intersection with open set theorem 1})
We note that the second assertion of part (\ref{intersection with open set theorem 1}) follows from the first, since a set with non-zero Hausdorff dimension is uncountable (see \cite[p.~49]{Fal14}).

To prove the first assertion, suppose that $F$ is uncountable. Then, $E$ is uncountable, and, in particular, $E \neq \varnothing$. We show that $I_\sigma \cap E \neq \varnothing$ for any $\sigma \in \bigcup_{n \in \mathbb{N}} \Sigma_n$. Suppose otherwise. Then, there exist $M \in \mathbb{N}$ and $\tau \coloneqq (\tau_n)_{n=1}^M \in \Sigma_M$ such that $I_\tau \cap E = \varnothing$. Observe that $I_{(1, \dotsc, M)} \cap E \neq \varnothing$. This is because, if $x \in E$, then, since $M \leq d_M(x) < d_{M+1}(x)$ by Proposition \ref{f is homeo}(\ref{f is homeo 1}), we have
\[
\langle 1, \dotsc, M, d_{M+1}(x), d_{M+2}(x), \dotsc \rangle_P \in I_{(1, \dotsc, M)} \cap E
\]
by finite replacement-invariance of $E$. This tells us that $\tau_M \neq M$. (Indeed, if $\tau_M = M$, then the fact that $\tau \in \Sigma_M$ leads us to deduce that $(\tau_n)_{n=1}^M = (n)_{n=1}^M$ by the definition of $\Sigma_M$. This would imply that $\varnothing \neq I_{(1, \dotsc, M)} \cap E = I_\tau \cap E = \varnothing$, which is a contradiction.) Hence, $\tau_M \geq M+1$ so that $\tau_M-M-1 \in [0, \infty)$. Write
\[
E = \bigcup_{k \in \mathbb{N}_0} Y_k \cup Z,
\]
where 
\begin{align*}
Y_k &\coloneqq \{ x \in [0,1] : d_{M+k}(x) \geq \tau_M+k \} \cap E, \quad \text{for } k \in \mathbb{N}_0, \\
Z &\coloneqq Z_{\tau_M-M-1}^{(M)} \cap E.
\end{align*}
Here, the set $Z_{\tau_M-M-1}^{(M)}$ is defined as in \eqref{definition of Z c}, i.e.,
\[
Z_{\tau_M-M-1}^{(M)} \coloneqq \{ x \in \mathbb{I} : d_n(x) \leq n + (\tau_M-M-1) \text{ for all } n \geq M \}.
\]
In fact, on one hand, it is clear from definitions that $Y_k \subseteq E$ for each $k \in \mathbb{N}_0$ and that $Z \subseteq E$; on the other hand, if $x \in E \setminus \bigcup_{k \in \mathbb{N}_0} Y_k$, then $d_{M+k}(x) \leq \tau_M+k-1$ for each $k \in \mathbb{N}_0$ so that $x \in Z_{\tau_M-M-1}^{(M)}$. Notice that each $Y_k$, $k \in \mathbb{N}_0$, is empty. Indeed, if $y \in Y_k$ for some $k \in \mathbb{N}_0$, then, since $\tau_{M}+k \leq d_{M+k}(y) < d_{M+k+1}(y)$, we have
\[
\langle \tau_1, \dotsc, \tau_M, \underbrace{\tau_M+1, \dotsc, \tau_M+k}_{\text{$k$ terms}}, d_{M+k+1}(y), d_{M+k+2}(y), \dotsc \rangle_P \in I_\tau \cap E,
\]
where the containment in $E$ is due to finite replacement-invariance of $E$. This contradicts the hypothesis $I_\tau \cap E = \varnothing$. Thus, $E = Z$, i.e., $E \subseteq Z_{\tau_M-M-1}^{(M)}$. But then, since $Z_{\tau_M-M-1}^{(M)}$ is countable by Lemma \ref{Z c is countable lemma}, it follows that $E$ is countable, contradicting the fact that $E$ is uncountable. Therefore, $I_\sigma \cap E \neq \varnothing$ for any $\sigma \in \bigcup_{n \in \mathbb{N}} \Sigma_n$, and thus, Lemma \ref{prevalence lemma} tells us that $E$ is dense in $[0,1]$.

(\ref{intersection with open set theorem 2})
By Lemma \ref{I sigma and open set equivalence lemma}, it suffices to show that $\hdim (I_\sigma \cap E) = \hdim E$ for any $\sigma \in \bigcup_{n \in \mathbb{N}} \Sigma_n$. Let $\sigma \coloneqq (\sigma_k)_{k=1}^n \in \Sigma_n$ for some $n \in \mathbb{N}$. If $\hdim E = 0$, then $\hdim (I_\sigma \cap E) =0$ by monotonicity (Proposition \ref{monotonicity and countable stability}(\ref{monotonicity and countable stability 1})), and hence, the equality $\hdim E = \hdim (I_\sigma \cap E)$ holds true.

Assume that $\hdim E > 0$. Put $\sigma_0 \coloneqq 0$. For each $m \in \{ 1, \dotsc, n \}$, define $\sigma^{(m)} (j) \in \Sigma_m$ by
\[
\sigma^{(m)} (j) \coloneqq (\underbrace{\sigma_1, \dotsc, \sigma_{m-1}}_{\text{$m-1$ terms}}, j)
\]
for each integer $j \geq \sigma_{m-1}+1$.

\begin{claim} 
For each $m \in \{ 1, \dotsc, n \}$, we have
\[
\hdim (I_{\sigma^{(m)} (j)} \cap E) = \hdim (I_{\sigma^{(m)} (j+1)} \cap E)
\]
for any integer $j \geq \sigma_{m-1}+1$.
\end{claim}

\begin{proof} [Proof of Claim] \renewcommand\qedsymbol{$\blacksquare$}
Let $\tau \in \bigcup_{n \in \mathbb{N}} \Sigma_n$ be arbitrary. Consider the map $g_\tau \colon [0,1] \to g_\tau ([0,1])$, defined as in Proposition \ref{shift of digits proposition}, and its inverse $g_\tau^{-1}$. By Proposition \ref{shift of digits proposition}(\ref{shift of digits proposition 2}), we have $I_\tau \cap E \subseteq I_\tau \cap \mathbb{I} \subseteq g_\tau([0,1])$, with $I_\tau \cap E$ being non-empty by part (\ref{intersection with open set theorem 1}) of this theorem and Lemma \ref{prevalence lemma}. So, we may consider the restriction of the map $g_\tau^{-1}$ to the set $I_\tau \cap E$. Throughout the proof of this claim, for each $\tau \in \bigcup_{n \in \mathbb{N}} \Sigma_n$, we define a map
\begin{align} \label{definition of g tilde}
\widetilde{g}_\tau \colon I_\tau \cap E \to g_\tau^{-1}(I_\tau \cap E) \quad \text{by} \quad \widetilde{g}_\tau \coloneqq g_\tau^{-1} \quad \text{on} \quad I_\tau \cap E. 
\end{align}
Note that, by definition, $\widetilde{g}_\tau$ is bi-Lipschitz and bijective.

Let $m \in \{ 1, \dotsc, n \}$, and let $j \in \mathbb{N}$ be arbitrary such that $j \geq \sigma_{m-1}+1$. We first show that
\[
\hdim (I_{\sigma^{(m)} (j)} \cap E) \geq \hdim (I_{\sigma^{(m)} (j+1)} \cap E).
\]
Define the map $\widetilde{g}_{\sigma^{(m)} (j+1)}$ as in \eqref{definition of g tilde}, which is bi-Lipschitz and bijective. Then, the composition
\[
g \coloneqq g_{\sigma^{(m)} (j)} \circ \widetilde{g}_{\sigma^{(m)} (j+1)} \colon I_{\sigma^{(m)} (j+1)} \cap E \to g (I_{\sigma^{(m)} (j+1)} \cap E)
\]
is well-defined, and is bi-Lipschitz and bijective. Observe that
\begin{align} \label{image of g is subset}
g (I_{\sigma^{(m)} (j+1)} \cap E) \subseteq I_{\sigma^{(m)} (j)} \cap E.
\end{align}
To see this, let $x \in I_{\sigma^{(m)} (j+1)} \cap E$, and write
\begin{align} \label{inverse image of x under phi 1}
x = \langle \underbrace{\sigma_1, \dotsc, \sigma_{m-1}}_{\text{$m-1$ terms}}, j+1, d_{m+1}(x), d_{m+2}(x), \dotsc \rangle_P.
\end{align}
By Proposition \ref{shift of digits proposition}(\ref{shift of digits proposition 2}), it follows that
\[
\widetilde{g}_{\sigma^{(m)} (j+1)}(x) = \langle d_{m+1}(x), d_{m+2}(x), \dotsc \rangle_P \in \mathbb{I}.
\]
Since $j < j+1 < d_{m+1}(x)$, we find, in view of Proposition \ref{shift of digits proposition}(\ref{shift of digits proposition 1}), that
\[
g(x) = 
\langle \underbrace{\sigma_1, \dotsc, \sigma_{m-1}}_{\text{$m-1$ terms}}, j, d_{m+1}(x), d_{m+2}(x), \dotsc \rangle_P.
\]
Hence, $g(x) \in I_{\sigma^{(m)} (j)} \cap E$, where the containment in $E$ follows from \eqref{inverse image of x under phi 1} and the finite replacement-invariance of $E$. Hence, we have the inclusion \eqref{image of g is subset}. Thus, by using bi-Lipschitz invariance (Proposition \ref{bi-Lipschitz invariance}), \eqref{image of g is subset}, and monotonicity (Proposition \ref{monotonicity and countable stability}(\ref{monotonicity and countable stability 1})), we conclude that
\[
\hdim (I_{\sigma^{(m)} (j+1)} \cap E) = \hdim g(I_{\sigma^{(m)} (j+1)} \cap E) \leq \hdim (I_{\sigma^{(m)} (j)} \cap E),
\]
as desired. 

It remains to prove the reverse inequality
\begin{align} \label{reverse inequality}
\hdim (I_{\sigma^{(m)} (j)} \cap E) \leq \hdim (I_{\sigma^{(m)} (j+1)} \cap E).
\end{align}
For each $k \geq 2$ and $(l, \tau_{m+1}, \dotsc, \tau_{m+k-1}) \in \Sigma_k$, where $l \in \{ j, j+1 \}$, define 
\[
\sigma^{(m)} (l, \tau_{m+1}, \dotsc, \tau_{m+k-1}) \coloneqq 
(\underbrace{\sigma_1, \dotsc, \sigma_{m-1}}_{\text{$m-1$ terms}}, l, \tau_{m+1}, \dotsc, \tau_{m+k-1}),
\]
which clearly satisfies $\sigma^{(m)} (l, \tau_{m+1}, \dotsc, \tau_{m+k-1}) \in \Sigma_{m+k-1}$ since $\sigma_{m-1} < j \leq l$. Write
\begin{align} \label{partition of I sigma m j}
I_{\sigma^{(m)} (j)} \cap E = \bigcup_{k \geq 2} Y_k \cup Z,
\end{align}
where
\begin{align*}
Y_k
&\coloneqq \bigcup_{\substack{\sigma^{(m)} (j, \tau_{m+1}, \dotsc, \tau_{m+k-1}) \in \Sigma_{m+k-1} \\ \tau_{m+k-1} \geq j+k}} [I_{\sigma^{(m)} (j, \tau_{m+1}, \dotsc, \tau_{m+k-1})} \cap E], \quad \text{for } k \geq 2, \\
Z
&\coloneqq I_{\sigma^{(m)} (j)} \cap Z_{j-m}^{(m+1)} \cap E.
\end{align*}
Here, the set $Z_{j-m}^{(m+1)}$ is defined as in \eqref{definition of Z c}, i.e.,
\[
Z_{j-m}^{(m+1)} \coloneqq \{ x \in \mathbb{I} : d_k(x) \leq k+(j-m) \text{ for all } k \geq m+1 \}.
\]
To see that \eqref{partition of I sigma m j} holds true, note first that the set on the right-hand side is evidently contained in the set on the left-hand side by the definition of the fundamental intervals. For the reverse inclusion, let $x \in I_{\sigma^{(m)} (j)} \cap E$, and suppose that $x \not \in \bigcup_{k \geq 2} Y_k$. Then, for each $k \geq 2$, since $x \not \in Y_k$, it must be that $d_{m+k-1}(x) \leq j+k-1$. Hence, $x \in Z_{j-m}^{(m+1)}$ so that $x \in Z$, and this verifies \eqref{partition of I sigma m j}.

Now, since $Z \subseteq Z_{j-m}^{(m+1)}$ by definition and since $\hdim Z_{j-m}^{(m+1)} = 0$ by Lemma \ref{Z c is countable lemma}, monotonicity (Proposition \ref{monotonicity and countable stability}(\ref{monotonicity and countable stability 1})) implies that $\hdim Z = 0$. Then, by \eqref{partition of I sigma m j} and countable stability (Proposition \ref{monotonicity and countable stability}(\ref{monotonicity and countable stability 2})), we deduce that
\begin{align*}
\hdim (I_{\sigma^{(m)} (j)} \cap E)
&= \sup_{k \geq 2} \{ \hdim Y_k, \hdim Z \} = \sup_{k \geq 2} \{ \hdim Y_k \}.
\end{align*}
Thus, to establish \eqref{reverse inequality}, it is enough to show that the inequality \eqref{hdim Y k upper bound} below holds for each integer $k \geq 2$.
\begin{align} \label{hdim Y k upper bound}
\hdim Y_k \
&= \sup_{\substack{\sigma^{(m)} (j, \tau_{m+1}, \dotsc, \tau_{m+k-1}) \in \Sigma_{m+k-1} \\ \tau_{m+k-1} \geq j+k}} \{ \hdim (I_{\sigma^{(m)} (j, \tau_{m+1}, \dotsc, \tau_{m+k-1})} \cap E) \} \nonumber \\
&\leq \hdim (I_{\sigma^{(m)} (j+1)} \cap E),
\end{align}
where the equality holds by the definition of $Y_k$ and countable stability (Proposition \ref{monotonicity and countable stability}(\ref{monotonicity and countable stability 2})). To achieve this, fix $k \geq 2$, and fix a sequence $\sigma^{(m)} (j, \tau_{m+1}, \dotsc, \tau_{m+k-1}) \in \Sigma_{m+k-1}$ satisfying $\tau_{m+k-1} \geq j+k$. Define the map $\widetilde{g}_{\sigma^{(m)} (j, \tau_{m+1}, \dotsc, \tau_{m+k-1})}$ as in \eqref{definition of g tilde}, which is bi-Lipschitz and bijective. Then, the composition
\begin{align*}
&h \coloneqq g_{\sigma^{(m)} (j+1, \dotsc, j+k)} \circ \widetilde{g}_{\sigma^{(m)} (j, \tau_{m+1}, \dotsc, \tau_{m+k-1})} \colon \\
&\hspace{3cm} I_{\sigma^{(m)} (j, \tau_{m+1}, \dotsc, \tau_{m+k-1})} \cap E \to h(I_{\sigma^{(m)} (j, \tau_{m+1}, \dotsc, \tau_{m+k-1})} \cap E)
\end{align*}
is well-defined, and is bi-Lipschitz and bijective. Notice that
\begin{align} \label{image of h is subset}
h (I_{\sigma^{(m)} (j, \tau_{m+1}, \dotsc, \tau_{m+k-1})} \cap E) \subseteq I_{\sigma^{(m)} (j+1)} \cap E.
\end{align}
To see this, let $x \in I_{\sigma^{(m)} (j, \tau_{m+1}, \dotsc, \tau_{m+k-1})} \cap E$, and write
\begin{align} \label{inverse image of x under phi 2}
x = \langle \underbrace{\sigma_1, \dotsc, \sigma_{m-1}}_{\text{$m-1$ terms}}, j, \tau_{m+1}, \dotsc, \tau_{m+k-1}, d_{m+k}(x), d_{m+k+1}(x), \dotsc \rangle_P.
\end{align}
By Proposition \ref{shift of digits proposition}(\ref{shift of digits proposition 2}), we have
\[
\widetilde{g}_{\sigma^{(m)} (j, \tau_{m+1}, \dotsc, \tau_{m+k-1})}(x) = \langle d_{m+k}(x), d_{m+k+1}(x), \dotsc \rangle_P \in \mathbb{I}.
\]
Since $\sigma_{m-1} < j < j+1$ and $j+k  \leq \tau_{m+k-1} < d_{m+k}(x)$, we infer, in light of Proposition \ref{shift of digits proposition}(\ref{shift of digits proposition 1}), that
\[
h(x) = \langle \underbrace{\sigma_1, \dotsc, \sigma_{m-1}}_{\text{$m-1$ terms}}, j+1, \dotsc, j+k, d_{m+k}(x), d_{m+k+1}(x), \dotsc \rangle_P.
\]
Then, $h(x) \in I_{\sigma^{(m)} (j+1)} \cap E$, where the containment in $E$ follows from \eqref{inverse image of x under phi 2} and finite replacement-invariance of $E$. Hence, we have the inclusion \eqref{image of h is subset}. Thus, by using bi-Lipschitz invariance (Proposition \ref{bi-Lipschitz invariance}), \eqref{image of h is subset}, and monotonicity (Proposition \ref{monotonicity and countable stability}(\ref{monotonicity and countable stability 1})), we find that
\begin{align*}
\hdim (I_{\sigma^{(m)} (j, \tau_{m+1}, \dotsc, \tau_{m+k-1})} \cap E) 
&= \hdim h (I_{\sigma^{(m)} (j, \tau_{m+1}, \dotsc, \tau_{m+k-1})} \cap E) \\
&\leq \hdim (I_{\sigma^{(m)} (j+1)} \cap E).
\end{align*}
This proves that \eqref{hdim Y k upper bound} holds for each integer $k \geq 2$. Hence, we have \eqref{reverse inequality}. This completes the proof of the claim.
\end{proof}

We are now ready to finish the proof of the theorem by using Claim. By \eqref{partition of unit interval} and countable stability (Proposition \ref{monotonicity and countable stability}(\ref{monotonicity and countable stability 2})), we have
\[
\hdim E
= \hdim \bigcup_{j \in \mathbb{N}} [I_{(j)} \cap E]
= \sup_{j \in \mathbb{N}} \{ \hdim (I_{(j)} \cap E) \}
= \hdim (I_{(\sigma_1)} \cap E),
\]
where we used Claim for the last equality. Recall that $E \subseteq \mathbb{I}$. Then, similarly, by using \eqref{partition of I sigma}, countable stability (Proposition \ref{monotonicity and countable stability}(\ref{monotonicity and countable stability 2})), and Claim, we find that
\begin{align*}
\hdim (I_{(\sigma_1)} \cap E)
&= \hdim \bigcup_{j \geq \sigma_1+1} [I_{(\sigma_1, j)} \cap E] \\
&= \sup_{j \geq \sigma_1+1} \{ \hdim (I_{(\sigma_1, j)} \cap E) \}
= \hdim (I_{(\sigma_1, \sigma_2)} \cap E).
\end{align*}
Therefore, using a similar argument, it can be concluded that
\begin{align*}
\hdim E
&= \hdim (I_{(\sigma_1)} \cap E)
= \hdim (I_{(\sigma_1, \sigma_2)} \cap E) 
= \dotsb = \hdim (I_{(\sigma_1, \sigma_2, \dotsc, \sigma_n)} \cap E),
\end{align*}
as was to be shown.
\end{proof}

\begin{proof} [Proof of Corollary \ref{LLN corollary}]
Let $\alpha \in [0, \infty]$. It is immediate from the definition \eqref{definition of A alpha} that $A(\alpha)$ is finite replacement-invariant. Recall from Proposition \ref{hdim A alpha} that $\hdim A(\alpha)=1>0$. We infer, in view of Theorem \ref{intersection with open set theorem}(\ref{intersection with open set theorem 1}), that $A(\alpha)$ is dense in $[0,1]$. Hence, we obtain part (\ref{LLN corollary 1}). Moreover, Theorem \ref{intersection with open set theorem}(\ref{intersection with open set theorem 2}) tells us that $\hdim (U \cap A(\alpha)) = \hdim A(\alpha)$ for any non-empty open subset $U$ of $[0,1]$. This establishes part (\ref{LLN corollary 2}).
\end{proof}

\subsection{Proofs of Theorem \ref{leap year theorem} and Corollary \ref{leap year corollary}}

Recall from \eqref{definition of S beta} the definition of the set $S(\alpha)$, $\alpha \in [0, \infty]$. The following lemma extends \cite[Theorem 3]{Sha94}, which presented the case for $\alpha=1$.

\begin{lemma} \label{A alpha subset S beta lemma}
For each $\alpha \in [0, \infty]$, we have $A(\alpha) \subseteq S(\alpha)$.
\end{lemma}

\begin{proof}
The inclusion $A(1) \subseteq S(1)$ was established by Shallit in \cite[Theorem 3]{Sha94}. For each $\alpha \in (0, \infty]$, by following a similar line to the proof of \cite[Theorem 3]{Sha94} and by adjusting a parameter, and employing Lemmas \ref{growth rate of log product lemma} and \ref{sum of reciprocal lemma}, it is not hard to derive the desired inclusion. Hence, we omit the details.

Now, suppose that $x \in A(0)$. We show that $x \in S(0)$, i.e.,
\[
\limsup_{N \to \infty} \frac{N x - L(f(x), N)}{\sqrt{\log N}} = \infty
\quad \text{and} \quad
\liminf_{N \to \infty} \frac{N x - L(f(x), N)}{\sqrt{\log N}} = -\infty.
\]
Let $\varepsilon \in (0, \infty)$ be arbitrary. Since $(\log d_n(x))/n \to 0$ as $n \to \infty$ by the hypothesis, we can find a $K \in \mathbb{N}$ such that $\log d_n(x) < n \varepsilon$ for all $n>K$. For each $j \in \mathbb{N}$, put
\[
N_j \coloneqq -1 + d_1(x) - d_1(x) d_2(x) + \dotsb + (-1)^{j+1} d_1(x) d_2(x) \dotsm d_j(x).
\]
Notice that
\begin{align} \label{lower bound of N odd}
1 < N_3 < N_5 < \dotsb
\quad \text{and} \quad
N_{2r+1} < \prod_{j=1}^{2r+1} d_j(x) \text{ for all } r \in \mathbb{N}.
\end{align}
Then, for any positive integer $r \geq K/2$, we have
\[
0 < \log N_3 \leq \log N_{2r+1} < \sum_{j=1}^{2r+1} \log d_j(x) < \sum_{j=1}^K \log d_j(x)  + \varepsilon \sum_{j=K+1}^{2r+1} j.
\]
By using \cite[Theorem 2]{Sha94}, which states that
\[
N_{2r+1} x - L(f(x), N_{2r+1}) \geq \frac{r}{4} > 0 \quad \text{for any } r \in \mathbb{N},
\]
we find that
\[
\frac{N_{2r+1}x - L(f(x), N_{2r+1})}{\sqrt{\log N_{2r+1}}} > \frac{r/4}{\sqrt{\sum_{j=1}^K \log d_j(x) + \varepsilon (2r-K+1)(2r+K+2)/2}}
\]
for any positive integer $r \geq K/2$. Since $(N_{2r+1})_{r \in \mathbb{N}}$ is a strictly increasing sequence of positive integers by \eqref{lower bound of N odd}, it follows that
\[
\limsup_{N \to \infty} \frac{N x - L(f(x), N)}{\sqrt{\log N}} \geq \limsup_{r \to \infty} \frac{N_{2r+1}x - L(f(x), N_{2r+1})}{\sqrt{\log N_{2r+1}}} \geq \frac{1}{4\sqrt{2 \varepsilon}}.
\]
By letting $\varepsilon \to 0^+$, we obtain the desired limsup. One can prove the liminf part in a similar way by using the sequence $(M_j)_{j \in \mathbb{N}}$ defined by $M_j \coloneqq - N_j$ for each $j \in \mathbb{N}$. We leave the detailed calculations as an exercise for the readers.
\end{proof}

\begin{proof} [Proof of Theorem \ref{leap year theorem}]
Let $\alpha \in [0, \infty]$. We have $A(\alpha) \subseteq S(\alpha)$ by Lemma \ref{A alpha subset S beta lemma}. Since $A(\alpha)$ is dense in $[0,1]$ by Corollary \ref{LLN corollary}(\ref{LLN corollary 1}), we conclude that $S(\alpha)$ is dense in $[0,1]$. This proves part (\ref{leap year theorem 1}). Now, let $U$ be a non-empty open subset of $[0,1]$. Then, by Lemma \ref{A alpha subset S beta lemma}, we have $U \cap A(\alpha) \subseteq U \cap S(\alpha)$. Here, $\hdim (U \cap A(\alpha)) = 1$ by Corollary \ref{LLN corollary}(\ref{LLN corollary 2}), and thus, $\hdim (U \cap S(\alpha)) = 1$ by monotonicity (Proposition \ref{monotonicity and countable stability}(\ref{monotonicity and countable stability 1})). This establishes part (\ref{leap year theorem 2}).
\end{proof}

\begin{remark}
We have not determined whether $S(\alpha)$, $\alpha \in [0, \infty]$, is finite replacement-invariant. Since we already know that $S(\alpha)$ contains some dense subset of $[0,1]$ with full Hausdorff dimension, such a determination was not necessary. Moreover, if $S(\alpha)$ turns out to be finite replacement-invariant, one must ascertain first that $S(\alpha)$ is uncountable or that $\hdim S(\alpha) \neq 0$ before applying Theorem \ref{intersection with open set theorem} directly to $S(\alpha)$ in proving Theorem \ref{leap year theorem}(\ref{leap year theorem 1}). For Theorem \ref{leap year theorem}(\ref{leap year theorem 2}), even more critically, one must ascertain the precise value of $\hdim S(\alpha)$.
\end{remark}

\begin{proof} [Proof of Corollary \ref{leap year corollary}]
Put $F \coloneqq [0,1] \setminus S(1)$, i.e., $F$ is the set of exceptions to \eqref{law of leap years}. Let $\alpha \in [0, \infty] \setminus \{ 1 \}$. By definition, it is clear that $S(\alpha) \subseteq F$. But, due to Theorem \ref{leap year theorem}, $S(\alpha)$ is dense in $[0,1]$ and $\hdim S(\alpha) = 1$. Therefore, $F$ is dense in $[0,1]$, and, by monotonicity (Proposition \ref{monotonicity and countable stability}(\ref{monotonicity and countable stability 1})), $\hdim F = 1$.
\end{proof}

\section*{Acknowledgements}

I would like to express my gratitude to the referees for their thorough review and thoughtful suggestions, which have greatly improved the readability of this manuscript. I am especially grateful for their insight that the assumption in Theorem \ref{intersection with open set theorem}(\ref{intersection with open set theorem 1}) can be weakened to uncountability, rather than the non-zero Hausdorff dimension used in the original manuscript.

\end{document}